\documentclass{article}
\usepackage{amsmath,amssymb,stmaryrd,enumerate,bbm,theorem}

\newcommand{\assign}{:=}
\newcommand{\nin}{\not\in}
\newcommand{\tmem}[1]{{\em #1\/}}
\newcommand{\tmop}[1]{\ensuremath{\operatorname{#1}}}
\newcommand{\tmstrong}[1]{\textbf{#1}}
\newcommand{\tmtextit}[1]{{\itshape{#1}}}
\newenvironment{descriptionlong}{\begin{description} }{\end{description}}
\newenvironment{enumeratealpha}{\begin{enumerate}[a{\textup{)}}] }{\end{enumerate}}
\newenvironment{enumerateroman}{\begin{enumerate}[i.] }{\end{enumerate}}
\newenvironment{itemizeminus}{\begin{itemize} }{\end{itemize}}
\newenvironment{proof}{\noindent\textbf{Proof\ }}{\hspace*{\fill}$\Box$\medskip}
\newtheorem{proposition}{Proposition}
{\theorembodyfont{\rmfamily}\newtheorem{remark}{Remark}}
\newtheorem{theorem}{Theorem}

\begin{document}

\title{Rare events, exponential hitting times and extremal indices via
spectral perturbation}\author{Gerhard Keller}\date{December 22, 2011}\maketitle

\begin{abstract}
  We discuss how an eigenvalue perturbation formula for
  transfer operators of dynamical systems is related to exponential hitting
  time distributions and extreme value theory for processes generated by
  chaotic dynamical systems. We also list a number of piecewise expanding
  systems to which this general theory applies and discuss the prospects to
  apply this theory to some classes of piecewise hyperbolic systems.
\end{abstract}

\section{Introduction}

The first occurence of a rare event $A_{\varepsilon}$ in a time
discrete dynamical system $T : M \rightarrow M$ can be described in
terms of the first hitting time $\tau_{\varepsilon} (x) = \inf \{i
\geqslant 0 : T^i x \in A_{\varepsilon} \}$ $(x \in M)$. In order to
justify the adjective ``rare'' one should assume that $\mu
(A_{\varepsilon})$ is small, where $\mu$ is a suitable reference
probability measure on $M$. For many dynamical systems it is then
known that the law of $\tau_{\varepsilon}$ under $\mu$ is nearly
exponential with a parameter $\lambda_{\varepsilon}$ that is closely
related to $\mu (A_{\varepsilon})$, see {\cite{coelho2000}} for an
excellent review. In this note we describe how asymptotic expressions
(as $\varepsilon\to0$) for the exponent $\lambda_{\varepsilon}$ of
exponential hitting time statistics can be derived from spectral
perturbation results for the Perron Frobenius operator of the
dynamical system, provided the operator has a spectral gap on a
suitable space of (generalized) functions.

\subsection{Perturbation theory for dynamical systems}

In 1999, C. Liverani and the author {\cite{KL1999}} published a spectral
perturbation theorem for linear operators that is particularly suited to deal
with transfer operators of dynamical systems: it does not require the
perturbation to be small in operator norm (which is rarely the case for
transfer operators if the underlying dynamics are perturbed). Of course there
is a price to pay for this: the operators must satisfy a uniform inequality
which is known as Lasota-Yorke inequality in dynamical systems or as
Doeblin-Fortet inequality in the theory of Markov operators, and the
conclusions are not as sharp as one might expect, namely the perturbed
spectral quantities do not depend Lipschitz-continuously on the size of the
perturbation.

There are in fact examples showing that the latter drawback of this
perturbation theorem is not just an artifact of the proof
{\cite[Theorem 6.1 and Footnote 7]{baladi2007}}. Therefore, a
perturbation result for the leading eigenvalue of transfer operators
published by C. Liverani and the author {\cite{KL2009}} in 2009 might
seem contradictory at first sight: under rather general assumptions
that are met in particular when the aforementioned spectral
perturbation theorem applies, this result shows that the leading
eigenvalue of a transfer operator is differentiable as a function of
the size of the perturbation at the unperturbed operator. This covers
the case where the leading eigenvalue becomes smaller than $1$ due to
a small leak in the phase space of the system so that this eigenvalue
determines the escape rate of the open system.

\subsection{Escape rates in open dynamical systems}

The study of escape rates (also called decay rates) in open dynamical systems
has different aspects. Starting with {\cite{PY1979,pianigiani1981}} there are
numerous publications highlighting the role these rates play for conditionally
invariant measures, see {\cite{coelho2000,DY2006}} for reviews and further
perspectives. In {\cite{keller1984}} and {\cite[section 9.6C]{keller1989}}
escape rates were characterized in terms of topological pressure, and
{\cite{DWY2011}} is a recent systematic account of this approach.

Here, on the contrary, the focus is on precise approximations of the escape
rate for small holes. This work was motivated by a result of Bunimovich and
Yurchenko {\cite[Theorem 4.6.1]{BY2011}} that I first learned about during a
conference at the Erwin-Schr\"odinger Institute in 2008. It gave rise to the
paper {\cite{KL2009}}. Since that paper was written, not only several other
authors considered related problems on escape rates and intermittent phenomena
{\cite{BB2010,BB2011,BV2011,DWY2011,FP2010,FMS2010,FS2010}}, but I also had
the opportunity to learn about the close relations between escape rates,
hitting and return time statistics and extremal indices for extreme value
distributions{\footnote{I want to mention in particular the talk by J. Freitas
``Laws of rare events for chaotic dynamical systems'' and stimulating
conversations with L. Bunimovich and M. Demers at the conference ``Large
deviations in dynamical systems'' at the CIRM in June 2011.}}
{\cite{FFT2009,FFT2010a,FFT2010b}}. In \ particular the extremal index turns
out to be a quantity that appears also in the eigenvalue perturbation formula
from {\cite{KL2009}}, and the sharp error term for exponential hitting times
that was proved for mixing processes in {\cite{abadi2004}} is also closely
related to the same formula.

\subsection{Outline of this paper}

The goal of this note is to provide details for the interrelations described
above in the particular case where the dynamics of open systems with small
holes are described by a perturbed Perron Frobenius operator. In particular,
\begin{itemizeminus}
  \item some details are given why the approximation formula for the leading
  eigenvalue is valid in the perturbative setting from {\cite{KL1999}}
  (section \ref{sec:perturb}),
  
  \item a slight extension of the arguments from {\cite{KL2009}} is given
  that provides estimates for the error term in the exponential approximation
  of the distribution of hitting times (section \ref{subsec:sharp}),
  
  \item it is argued that exponential hitting time statistics (including a
  formula for the extremal index) follow from the formula for the leading
  eigenvalue (section \ref{subsec:extreme}),
  
  \item and a number of concrete settings are listed to which the perturbation
  formula applies (section \ref{sec:review}).
\end{itemizeminus}
\section{A review of some perturbation results}\label{sec:perturb}

The paper {\cite{KL2009}} provides a first order approximation (as
$\varepsilon \rightarrow 0$) for the leading eigenvalues of the
Perron-Frobenius operator of an open dynamical system with a hole of ``size''
$\varepsilon$. The main result is formulated as a perturbation theorem for the
leading eigenvalue of a family $(P_{\varepsilon})_{\varepsilon \in E}$ of
linear operators on a Banach space $(V, \|.\|)$, where $E$ is a set of
parameters (equipped with some topology) containing a parameter called $0$
such that $P_0$ describes the dynamics without a hole. This theorem applies
whenever the following assumptions (\ref{eq:A1}) -- (\ref{eq:A6}) can be
verified:

There are $\lambda_{\varepsilon} \in \mathbbm{C}$, $\varphi_{\varepsilon} \in
V$, a bounded linear functional $\nu_{\varepsilon}:V\to{\mathbbm C}$, and linear operators $Q_{\varepsilon} : V \to
V$ such that
\begin{equation}
  \lambda_{\varepsilon}^{- 1} P_{\varepsilon} = \varphi_{\varepsilon} \otimes
  \nu_{\varepsilon} + Q_{\varepsilon}\text{ (where we assume $\lambda_0=1$),} \label{eq:A1}
\end{equation}
\begin{equation}
  P_{\varepsilon} (\varphi_{\varepsilon}) = \lambda_{\varepsilon}
  \varphi_{\varepsilon}, \hspace{0.75em} \nu_{\varepsilon} P_{\varepsilon} =
  \lambda_{\varepsilon} \nu_{\varepsilon}, \hspace{0.75em} Q_{\varepsilon}
  (\varphi_{\varepsilon}) = 0, \hspace{0.75em} \nu_{\varepsilon}
  Q_{\varepsilon} = 0, \label{eq:A2}
\end{equation}
\begin{equation}
  \sum_{n = 0}^{\infty} \sup_{\varepsilon \in E} \|Q_{\varepsilon}^n \|= : C_1
  < \infty, \label{eq:A3}
\end{equation}
Observe that assumptions (\ref{eq:A1}) and (\ref{eq:A2}) imply
$\nu_{\varepsilon} (\varphi_{\varepsilon}) = 1$ for all
$\varepsilon$.{\footnote{Assumption (\ref{eq:A3}) is clearly equivalent to the
seemingly stronger assumption that $\sup_{\varepsilon \in E}
\|Q_{\varepsilon}^n \|$ decays exponentially as $n \rightarrow \infty$. The
form of assumption (\ref{eq:A3}) chosen here underlines however, that formula
(\ref{eq:KL2009simple}) can be proved as well under slightly different
assumptions that allow in particular to use a weaker operator norm in
(\ref{eq:A3}) which is not submultiplicative, see {\cite[Remark
2.2]{KL2009}}.}} The fourth assumption relates the ``size'' of
$\varphi_{\varepsilon}$ to that of $\varphi_0$:
\begin{equation}
  \exists C_2 > 0 \forall \varepsilon \in E : \nu_0 (\varphi_{\varepsilon}) =
  1 \hspace{1em} \text{and \ \ } \| \varphi_{\varepsilon} \| \leqslant C_2 <
  \infty . \label{eq:A4}
\end{equation}
Now let
\[ \Delta_{\varepsilon} \assign \nu_0 (P_0 - P_{\varepsilon}) (\varphi_0) ,
\]
where $\nu_0(P_0-P_\epsilon)$ denotes the composition $\nu_0\circ(P_0-P_\epsilon):V\to{\mathbbm C}$.
Then the last two assumptions are
\begin{equation}
  \lim_{\varepsilon \rightarrow 0} \| \nu_0 (P_0 - P_{\varepsilon})\|= 0
  \label{eq:A5},
\end{equation}
\begin{equation}
  \| \nu_0 (P_0 - P_{\varepsilon})\| \cdot \|(P_0 - P_{\varepsilon}) \varphi_0
  \| \leqslant \tmop{const} \cdot | \Delta_{\varepsilon} | \label{eq:A6}
\end{equation}
Under these assumptions the main result of {\cite{KL2009}} is the formula
\begin{equation}
  1 - \lambda_{\varepsilon} = \Delta_{\varepsilon} \theta (1 + o(1)) \text{ \ \ in the limit \ } \varepsilon \rightarrow 0
  \label{eq:KL2009simple}
\end{equation}
where $\theta$ is a constant that takes care of short time correlations, see
section \ref{sec:KL2009} for more details.

\subsection{Verifying assumptions (\ref{eq:A1}) -- (\ref{eq:A4})}

One way to verify assumptions (\ref{eq:A1}) - (\ref{eq:A4}) that was already
taken in {\cite{KL2009}} and {\cite{FP2010}}, is to use the spectral
perturbation theorem from {\cite{KL1999}}. It holds (not only) under the
following assumptions:
\begin{enumerateroman}
  \item There are constants $\alpha\in(0,1)$, $D > 0$ and a second norm $| . |_w \leqslant
  \|.\|$ on $V$ (it is enough that this is a seminorm) such that:
  \begin{equation}
    \text{The residual spectrum of each $P_{\varepsilon}$ is contained in $\{z
    \in \mathbbm{C}: |z| \leqslant \alpha\}$. } \label{eq:B3}
  \end{equation}
  \begin{equation}
    \forall \varepsilon \in E\ \forall f \in V\ \forall n \in \mathbbm{N}:
    |P_{\varepsilon}^n f|_w \leqslant D\,|f|_w \label{eq:B1}
  \end{equation}
  \begin{equation}
    \forall \varepsilon \in E\ \forall f \in V\
    \forall n \in \mathbbm{N}: \|P_{\varepsilon}^n f\| \leqslant D \alpha^n
    \|f\|+ D|f|_w \label{eq:B2}
  \end{equation}
  \item $\text{There is a monotone upper semi-continuous function } \pi : E
  \rightarrow [0, \infty) \text{ such that }$
  \begin{equation}
    \lim_{\varepsilon \rightarrow 0} \pi_{\varepsilon} = 0 \text{ \ \ and \ \
    } \forall f \in V\ \forall \varepsilon \in E : |P_{\varepsilon} f - P_0
    f|_w \leqslant \pi_{\varepsilon} \|f\|. \label{eq:B4}
  \end{equation}
\end{enumerateroman}
An equivalent way to rewrite the last condition is:
\[ \forall \varepsilon \in E : \interleave P_{\epsilon} - P_0 \interleave
   \leqslant \pi_{\varepsilon} \]
where $\interleave R \interleave \assign \sup \{|R f|_w : f \in V, \|f\|
\leqslant 1\}$ for any linear operator $R : V \rightarrow V$.\\

The main result of {\cite{KL1999}} is the following theorem:

\begin{theorem}
  \label{theo:KL99}Suppose that $(P_{\varepsilon})_{\varepsilon \in E}$ is a
  family of linear operators on $V$ satisfying assumptions
  (\ref{eq:B3})-(\ref{eq:B4}). Fix $\delta > 0$ and $r \in (\alpha, 1)$, and
  let $\rho \assign \log \frac{r}{\alpha} / \log \frac{1}{\alpha} \in (0, 1)$.
  Then there are constants $a = a (r) > 0$, $b = b (\delta, r) > 0$, $c = c
  (\delta, r) > 0$ and $d = d (\delta, r) > 0$ and a neighborhood $E_{\delta,
  r}$ of $0$ in $E$ such that for $\varepsilon \in E_{\delta, r}$ and $z \in
  \mathbbm{C}$ with $|z| > r$ and $\tmop{dist} (z, \sigma (P_0)) > \delta$,
  \begin{equation}
    \|(z - P_{\varepsilon})^{- 1} f\| \leqslant a\|f\|+ b|f|_w  \text{ for all
    } f \in V \label{eq:KL99-1}
  \end{equation}
  and
  \begin{equation}
    \interleave (z - P_{\varepsilon})^{- 1} - (z - P_0)^{- 1} \interleave
    \leqslant \pi_{\varepsilon}^{\rho} \cdot \left( c\|(z - P_0)^{- 1} \|+
    d\|(z - P_0)^{- 1} \|^2 \right) . \label{eq:KL99-2}
  \end{equation}
\end{theorem}

\begin{remark}
  In {\cite{KL1999}} it is assumed that $E = [0, \infty)$, but as the
  parameter $\varepsilon$ enters the proofs only via the numbers
  $\pi_{\varepsilon}$, the same proof applies to a more general set $E$ of
  parameters. Similarly, in \ {\cite{KL2009}} it is assumed that $E$ is a
  closed subset of $\mathbbm{R}$, but again the parameter $\varepsilon$ enters
  the estimates only via the derived quantities $\eta_{\varepsilon}$ and
  $\Delta_{\varepsilon}$ defined in section \ref{sec:KL2009} below, so that
  also this result is valid for more general sets of parameters.
\end{remark}

\begin{remark}
  If assumptions (\ref{eq:B1}) and (\ref{eq:B2}) are met and if the closed
  unit ball of $(V, \|.\|)$ is $| . |_w$ -compact, then $\sigma
  (P_{\varepsilon}) \setminus \{z \in \mathbbm{C}: |z| \leqslant \alpha\}$
  consists only of isolated eigenvalues of finite multiplicity so that
  assumption (\ref{eq:B3}) is automatically satisfied. This is guaranteed by a
  strengthened version of the Ionescu-Tulcea/Marinescu theorem (see e.g.
  {\cite{HH2001}}).
  
  There are numerous situations where this remark applies:
  \begin{itemizeminus}
    \item $V$ is a space of H\"older continuous functions on a compact metric
    space and $| . |_w$ is the supremum norm. If the underlying space is not
    compact, weighted norms may be used; reference {\cite{HH2001}} provides
    many examples. In any case this setting is adapted to the study of Markov
    operators acting on observables and to Ruelle transfer operators, see
    Ruelle's monograph {\cite[Proposition 5.24]{ruelle1978}} and
    {\cite{CS2009}} for a recent contribution. As this setting is less useful
    for studying open systems, no further details are provided here.
    
    \item $V$ is the space of functions of bounded variation on a regular
    domain of finite volume in $\mathbbm{R}^d$ and $| . |_w$ is the
    $L^1_{\tmop{Leb}}$-seminorm. This setting is used for transfer operators of
    piecewise expanding maps, starting with
    {\cite{keller79,keller80,HK1982,rychlik83}} and being continued in
    numerous papers, see {\cite{baladi2000}} for an overview until 2000,
    {\cite{DL2008}} for references until 2008 and {\cite{thomine2011}} for a
    very recent contribution to this direction. In practically all variants of
    this setting the passage to an open system can be described as a small
    perturbation of the Perron-Frobenius operator in the sense of the
    $\interleave . \interleave$-norm.
    
  \item More generally, $V$ is a space of distributions on a
    finite-dimensional smooth manifold, equipped with a suitable norm,
    and $| . |_w$ is another weaker seminorm on $V$. Starting with
    {\cite{BKL2002}}, such spaces have been applied so far to various
    hyperbolic or piecewise hyperbolic systems, see
    e.g. {\cite{GL2006,BT2007,DL2008,BG2009,BG2010}}.
    
  \item $V$ is a space of distributions on $[0, 1]^{\mathbbm{Z}^d}$
    which contains in particular the space of signed measures with
    absolutely continuous marginals of (uniformly) bounded variation
    (see {\cite{KL2006,BGK2007,KL2009a}} for details). This setting was used
    to study weakly coupled lattices of piecewise expanding interval
    maps. Inequality (\ref{eq:B2}) was stated explicitly in
    {\cite[Lemma 3.4]{BGK2007}}.
  \end{itemizeminus}
  \begin{remark}
    The conclusions of Theorem \ref{theo:KL99} remain valid, if assumption
    (\ref{eq:B3}) is replaced by the following: If $\delta > 0$ and $r -
    \alpha > 0$ are sufficiently small, then the set
    \[ \{z \in \mathbbm{C}: |z| > r, \tmop{dist} (z, \sigma (P_0) > \delta\}
       \text{ \ is connected.,}  \]
    see {\cite[Remark 6]{KL1999}}.
  \end{remark}
\end{remark}

Now we specialize to the case where the following operator mixing condition
holds:
\begin{equation}
  \text{$P_0 = \varphi_0 \otimes \nu_0 + Q_0$ as in (\ref{eq:A1}) where $Q_0$
  has spectral radius strictly smaller than $1$.} \label{eq:B5} 
\end{equation}
Observe that this implies the (\ref{eq:B3}) for some $\alpha\in(0,1)$. 
Therefore, assumptions (\ref{eq:B1}) - (\ref{eq:B4}) and (\ref{eq:B5}) imply
the validity of (\ref{eq:A1}) - (\ref{eq:A3}). This follows easily
from Theorem \ref{theo:KL99}, see {\cite[Corollaries 1 and
  2]{KL1999}}. It remains to show that also (\ref{eq:A4}) follows:
Observe first that by {\cite[Corollary 1]{KL1999}},
\[ | \varphi_{\varepsilon} - \varphi_0 \cdot \nu_0 (\varphi_{\varepsilon}) |_w
   = | \varphi_{\varepsilon} \otimes \nu_{\varepsilon} (\varphi_{\varepsilon})
   - \varphi_0 \otimes \nu_0 (\varphi_{\varepsilon}) |_w \leqslant K_1
   \pi_{\varepsilon}^{\rho} \cdot \| \varphi_{\varepsilon} \| \leqslant K_1
   K_2 \pi_{\varepsilon}^{\rho} \cdot | \varphi_{\varepsilon} |_w \]
for suitable $K_1, K_2 > 0$ and $\rho$ from Theorem \ref{theo:KL99}. Hence
\[ | \varphi_{\varepsilon} |_w \cdot \left( 1 - K_1 K_2
   \pi_{\varepsilon}^{\rho} \right) \leqslant | \varphi_0 |_w \cdot | \nu_0
   (\varphi_{\varepsilon}) | \]
so that $| \varphi_{\varepsilon} |_w \leqslant 2 \cdot | \varphi_0 |_w \cdot |
\nu_0 (\varphi_{\varepsilon}) |$ when $\varepsilon \in E$ is sufficiently
close to $0$. Invoking {\cite[Corollary 1]{KL1999}} once more, this yields
\[ \| \varphi_{\varepsilon} \| \leqslant 2 K_2 \cdot | \varphi_0 |_w \cdot |
   \nu_0 (\varphi_{\varepsilon}) | \text{ \ \ \ for $\varepsilon \in E$ sufficiently
   close to } 0. \]
Therefore one can normalize $\varphi_{\varepsilon}$ (and in consequence
$\nu_{\varepsilon}$ as well) such that $\nu_0 (\varphi_{\varepsilon}) = 1$,
and (\ref{eq:A4}) holds for these normalized $\varphi_{\varepsilon}$ with a
constant $C_2 = 2 K_2 | \varphi_0 |_w$.

\subsection{Verifying assumptions (\ref{eq:A5}) and (\ref{eq:A6})}

In the remainder of this paper we restrict attention to the following
dynamical setting:

\begin{descriptionlong}
  \item[REPFO (Rare event Perron-Frobenius operators)] $V$ is a space of
  functions or distributions on a metric space $M$ that contains the constant
  functions, and $P_0$ is the Perron-Frobenius operator of a dynamical system
  $T : M \rightarrow M$ w.r.t. the reference measure $\nu_0$. It is mixing in
  the sense of (\ref{eq:B5}). The family $(A_{\varepsilon}$: $\varepsilon \in
  E)$ is a family of subsets of $M$ which are sufficiently regular such that
  \begin{enumerateroman}
    \item the operators $P_{\varepsilon}$, $P_{\varepsilon} f = P_0 (1_{M
    \setminus A_{\varepsilon}} f)$, satisfy assumptions (\ref{eq:B3}) --
    (\ref{eq:B4}) and (\ref{eq:B5}),
    
    \item $1_{A_{\varepsilon}} \cdot f \in V$ for all $f \in V$, and
    
    \item $| \nu_0 (1_{A_{\varepsilon}} f) | \cdot \|1_{A_{\varepsilon}}
    \varphi_0 \| \leqslant \tmop{const} \|f\| \cdot | \nu_0
    (1_{A_{\varepsilon}} \varphi_0) |$ for all $f \in V$.
  \end{enumerateroman}
\end{descriptionlong}

\begin{remark}
  If $V$ is a space of functions such that $\sup |f| \leqslant \|f\|$ for all
  $f \in V$, if $\nu_0$ is a measure on $M$, and if $\inf_{x \in
  \bigcup_{\varepsilon} A_{\varepsilon}} \varphi_0 (x) > 0$, then REPFO.iii is
  satisfied.
\end{remark}

\begin{remark}
  As $\nu_0 P_0 = \nu_0$, condition REPFO.iii implies (in fact, is nearly
  equivalent to)
  \[ \| \nu_0 (P_0 - P_{\varepsilon})\| \cdot \|(P_0 - P_{\varepsilon})
     (\varphi_0)\| \leqslant | \nu_0 (P_0 - P_{\varepsilon}) (\varphi_0) | . \]
  Let $\mu_0$ be the probability measure on $M$ defined by $\mu_0 (A) = \int_A
  \varphi_0 d \nu_0$. (It is the only $T$-invariant probability measure
  absolutely continuous w.r.t. $\nu_0$.) Then
  \[ \mu_0 (A_{\varepsilon}) = \nu_0 (P_0 - P_{\varepsilon}) (\varphi_0)
     \leqslant \tmop{const} \| \nu_0 (P_0 - P_{\varepsilon})\|. \]
\end{remark}

\begin{proposition}
  In a REPFO setting, assumptions (\ref{eq:A1})--(\ref{eq:A6}) are satisfied.
\end{proposition}

\begin{proof}
  In the preceding section it was argued that assumptions (\ref{eq:A1}) --
  (\ref{eq:A4}) are satisfied. Next observe that for $f \in V$ with $\|f\|
  \leqslant 1$,
  \begin{eqnarray*}
    \| \varphi_0 \| \cdot | \nu_0 (1_{A_{\varepsilon}} f) | & = & \|(\varphi_0
    \otimes \nu_0) (1_{A_{\varepsilon}} f)\|\\
    & = & \lim_{n \rightarrow \infty} \|P_0^n ((P_0 - P_{\varepsilon}) (f))\|
    \text{ \ \ \ \ \ \ \ \ \ \ \ \ \ \ \ by (\ref{eq:A3})}\\
    & \leqslant & D_3 | (P_0 - P_{\varepsilon}) (f) |_w \text{ \ \ \ \ \ \ \
    \ \ \ \ \ \ \ \ \ \ \ \ \ \ \ \ by ($\ref{eq:B2})$}\\
    & \leqslant & \pi_{\varepsilon} D_3 \|f\| \leqslant \pi_{\varepsilon}
    D_3 . \text{ \ \ \ \ \ \ \ \ \ \ \ \ \ \ \ \ \ \ \ \ \ \ \ \ by
    (\ref{eq:B4})}
  \end{eqnarray*}
  Therefore,
  \[ \| \nu_0 (P_0 - P_{\varepsilon})\|= \sup_{\|f\| \leqslant 1} | \nu_0
     (1_{A_{\varepsilon}} f) | \leqslant \pi_{\varepsilon}  \frac{D_3}{\|
     \varphi_0 \|} \]
  and
  \begin{eqnarray*}
    \| \nu_0 (P_0 - P_{\varepsilon})\| \cdot \|(P_0 - P_{\varepsilon})
    \varphi_0 \| & \leqslant & \tmop{const} \sup_{\|f\| \leqslant 1} | \nu_0
    (1_{A_{\varepsilon}} f) | \cdot \|1_{A_{\varepsilon}} \varphi_0 \| \text{
    \ \ \ \ \ \ \ \ \ by (\ref{eq:B2})}\\
    & \leqslant & \tmop{const} | \nu_0 (1_{A_{\varepsilon}} \varphi_0) |
    \text{ \ \ \ \ \ \ \ \ \ \ \ \ \ \ \ \ \ \ \ \ \ \ \ \ \ \ \ \ by iii.}\\
    & = & |\nu_0 (P_0 - P_\varepsilon)(\varphi_0)|.
  \end{eqnarray*}
  This proves (\ref{eq:A5}) and (\ref{eq:A6}).
\end{proof}
\quad\\
In section \ref{sec:review}, we discuss examples of REPFO settings.

\section{Rare events, exponential hitting time distributions, and extreme
value statistics}\label{sec:main}

Recall that we restrict to the REPFO setting in the remainder of this paper.

Our approach to estimate first hitting time distributions of rare events
proceeds along the following well known lines: denote the first hitting time
to $A_{\varepsilon}$ by $\tau_{\varepsilon}$, i.e. let $\tau_{\varepsilon} (x)
= \inf \{i \geqslant 0 : T^i x \in A_{\varepsilon} \}$ for $x \in M$. We are
interested in estimating $\nu_0 \{\tau_{\varepsilon} \geqslant n\}$ and $\mu_0
\{\tau_{\varepsilon} \geqslant n\}$.

If $\varphi \in V \tmop{is} \tmop{such} \tmop{that} \varphi \nu_0$ is a
probability measure on $M$, then
\begin{eqnarray}
  &  &  \nonumber\\
  \int_{\{\tau_{\varepsilon} \geqslant n\}} \varphi d \nu_0 & = & \int_M
  \prod_{i = 0}^{n - 1} 1_{M \setminus A_{\varepsilon}} \circ T^i \cdot
  \varphi d \nu_0 \nonumber\\
  & = & \int_M P_0^n \left( \prod_{i = 0}^{n - 1} 1_{M \setminus
  A_{\varepsilon}} \circ T^i \cdot \varphi \right) d \nu_0 \nonumber\\
  & = & \int_M P_{\varepsilon}^n \varphi\, d \nu_0 \nonumber\\
  & = & \lambda_{\varepsilon}^n  \int_M \varphi_{\varepsilon} \otimes
  \nu_{\varepsilon} (\varphi)\, d \nu_0 + \lambda_{\varepsilon}^n  \int_M
  Q_{\varepsilon}^n \varphi\, d \nu_0 \nonumber\\
  & = & \lambda_{\varepsilon}^n \nu_{\varepsilon} (\varphi)
  +\mathcal{O}(\lambda_{\varepsilon}^n \|Q_{\varepsilon}^n \| \cdot \| \varphi
  \|)  \label{eq:pineps}
\end{eqnarray}
Here the constant implied by ,,$\mathcal{O}$`` is uniform in $n$ and
$\varepsilon$. Applied with $\varphi = 1$ or $\varphi = \varphi_0$ one obtains
expressions for \ $\nu_0 \{\tau_{\varepsilon} \geqslant n\}$ and $\mu_0
\{\tau_{\varepsilon} \geqslant n\}$, respectively. To make use of this
identity, we need some control of $\lambda_{\varepsilon}$ as a function of
$\varepsilon$.

\subsection{The perturbation formula for the leading
eigenvalue}\label{sec:KL2009}

The main result of {\cite{KL2009}} asserts that $\lambda_{\varepsilon} = \exp
\left( - \mu_0 (A_{\varepsilon}) \theta (1 + o(1)) \right)$ with a
constant $\theta \in (0, 1]$ determined by the dynamics and a $o(1)$-term for $\varepsilon \rightarrow 0$.

For the further discussion we recall a more precise intermediate identity
from {\cite{KL2009}}. Let
\begin{eqnarray}
  \eta_{\varepsilon} & \assign & \| \nu_0 (P_0 - P_{\varepsilon})\|= \sup
  \left\{ \int_{A_{\varepsilon}} f\, d \nu_0 : f \in V, \|f\| \leqslant 1
  \right\}, \nonumber\\
  \Delta_{\varepsilon} & \assign & \nu_0 (P_0 - P_{\varepsilon}) (\varphi_0) =
  \int_{A_{\varepsilon}} \varphi_0\, d \nu_0 = \mu_0 (A_{\varepsilon}),
  \nonumber\\
  \kappa_N & \assign & \sum_{k = N}^{\infty} \sup_{\varepsilon \in E}
  \|Q_{\varepsilon}^k \|= \mathcal{O} ((1 - \gamma)^N),  \label{eq:kappaN}
\end{eqnarray}
where $\gamma > 0$ is some lower bound for the spectral gaps of the
$P_{\varepsilon}$. (It is strictly positive because of Theorem
\ref{theo:KL99}, see also {\cite[Corollary 2]{KL1999}}.) Then we have, for
each $\varepsilon \in E$ with $\Delta_{\varepsilon} \neq 0$ and for each $N \in
\mathbbm{N}$,
\begin{equation}
  \frac{1 - \lambda_{\varepsilon}}{\Delta_{\varepsilon}} = \left( \theta_{N,
  \varepsilon} + \mathcal{O} ((1-\gamma)^N) \right)  \left( 1 + \mathcal{O} (N
  \eta_{\varepsilon}) \right) \label{eq:KL09}
\end{equation}
with constants independent of $N$ and $\varepsilon$ \ {\cite[eq.
(6.2)]{KL2009}}{\footnote{If $\Delta_{\varepsilon} = 0$, the numbers $q_{k,
\varepsilon}$ are not well defined. Therefore the argument, given in
{\cite[after eq.(6.2)]{KL2009}}, that $\lambda_{\varepsilon} = \lambda_0$ in
this case, is wrong. The claim is nevertheless true, because assumption
(\ref{eq:A6}) implies that $\nu_0 (P_0 - P_{\varepsilon}) = 0$ or $(P_0 -
P_{\varepsilon}) (\varphi_0) = 0$. In both cases, $(\lambda_0 /
\lambda_{\varepsilon})^n = (\lambda_0 / \lambda_{\varepsilon})^n \nu_0
(\varphi_0) = \nu_0 (\lambda_{\varepsilon}^{- 1} P_{\varepsilon})^n
(\varphi_0) = \nu_0 (\varphi_{\varepsilon}) \nu_{\varepsilon} (\varphi_0) +
o(1) \neq 0$ when $n \rightarrow \infty$ in view of assumptions
(\ref{eq:A1}) -- (\ref{eq:A3}), so that $\lambda_{\varepsilon} = \lambda_0$ in
view of (\ref{eq:A4}).}}, where
\begin{equation}
  \theta_{N, \varepsilon} = 1 - \sum_{k = 0}^{N - 1} \lambda_{\varepsilon}^{-
  k} q_{k, \varepsilon} \text{ \ \ with \ \ } q_{k, \varepsilon} = \frac{\nu_0
  ((P_0 - P_{\varepsilon}) P_{\varepsilon}^k (P_0 - P_{\varepsilon})
  (\varphi_0))}{\Delta_{\varepsilon}} . \label{eq:xiNeps}
\end{equation}
For REPFO-type $P_{\varepsilon}$ we are looking at here this simplifies,
in case $\mu_0 (A_{\varepsilon}) > 0$, to
\begin{equation}
  q_{k, \varepsilon} = \frac{1}{\mu_0 (A_{\varepsilon})} 
  \int_{A_{\varepsilon}} P_{\varepsilon}^k P_0 (1_{A_{\varepsilon}} \varphi_0)
  d \nu_0 = \frac{\mu_0 (A_{\varepsilon} \cap T^{- 1} A_{\varepsilon}^c
  \cap T^{- 2} A_{\varepsilon}^c \cap \ldots \cap T^{- k} A_{\varepsilon}^c \cap T^{- (k+1)}
  A_{\varepsilon})}{\mu_0 (A_{\varepsilon})}.
\end{equation}

Observe that $\sum_{k = 0}^{\infty} q_{k, \varepsilon} = 1$ for each
$\varepsilon \in E$ by Kac's recurrence theorem. On the other hand, assuming
that
\begin{equation}
  q_k \assign \lim_{\varepsilon \rightarrow 0} q_{k, \varepsilon} \text{ \
  exists for all } k \text{ and denoting } \theta \assign 1 - \sum_{k =
  0}^{\infty} q_k \label{eq:xi},
\end{equation}
this latter sum can attain any value between $0$ and $1$. In any case, if these limits exist, and if we denote
$\theta_N = 1 - \sum_{k = 0}^{N - 1} q_k$, we get
\[ \lim_{\varepsilon \rightarrow 0} \frac{1 - \lambda_{\varepsilon}}{\mu_0
   (A_{\varepsilon})} = \theta_N + \mathcal{O} ((1 - \gamma)^N) . \]
As this is true for every $N$, we conclude, as in {\cite{KL2009}}, that
\begin{equation}
  \lim_{\varepsilon \rightarrow 0} \frac{1 - \lambda_{\varepsilon}}{\mu_0
  (A_{\varepsilon})} = \theta \assign 1 - \sum_{k = 0}^{\infty} q_k,
  \label{eq:KL2009b}
\end{equation}
i.e.
\begin{equation}
  \lambda_{\varepsilon} = 1 - \theta \mu_0 (A_{\varepsilon}) + o(\mu_0 (A_{\varepsilon})) = \exp (- \theta \mu_0 (A_{\varepsilon}) +
  o(\mu_0 (A_{\varepsilon}))) . \label{eq:lambdaeps}
\end{equation}
In the next subsections we will draw some finer conclusions on the
(non)-hitting probabilities $\nu_0 \{\tau_{\varepsilon} \geqslant n\}$ and
$\mu_0 \{\tau_{\varepsilon} \geqslant n\}$ working directly with equation
(\ref{eq:KL09}).

\begin{remark}
  For the map $T x = 2 x \tmop{mod} 1$ and $\nu_0$ the Lebesgue measure on
  $[0, 1]$, formula (\ref{eq:KL2009b}) was derived by combinatorial methods in
  {\cite[Theorem 4.6.1]{BY2011}}. In the context of Gibbs measures with a
  H\"older continuous potential this formula was derived also in
  {\cite{FP2010}}, using spectral theoretic methods and ideas on exponential
  first return times from {\cite{hirata1993}}. 
\end{remark}

\begin{remark}
  \label{remark:periodic1}If the map $T$ is continuous, if the holes
  $A_{\varepsilon}$ shrink to a single point $x \in M$ and if $\mu_0
  (A_{\varepsilon}) > 0$ for all $\varepsilon \neq 0$, then
  \begin{enumerateroman}
    \item $q_k = \lim_{\varepsilon \rightarrow 0} \frac{\mu_0 (A_{\varepsilon}
    \cap T^{- (k + 1)} A_{\varepsilon})}{\mu_0 (A_{\varepsilon})}$ if $x$ is a
    periodic point of minimal period $k + 1$,
    
    \item $q_k = 0$ otherwise.
  \end{enumerateroman}
  Therefore, $\theta < 1$ is possible only if $x$ is a periodic point. If $x$
  is a hyperbolic periodic point, then \ $\theta < 1$, indeed.
\end{remark}

  \begin{remark}\label{remark:rychlik1}
    Let $T:[0,1]\to[0,1]$ be a \emph{piecewise monotone interval map of
    Rychlik type} (which always fits the REPFO setting, see section
    \ref{subsec:rychlik}).  Denote by $\nu_0$ the Lebesgue measure on
    $[0,1]$. Assume furthermore that $V$ is a finite subset of
    $(0,1)$, that $A_\varepsilon$ is the $\varepsilon$-neighbourhood
    of $V$, and that the invariant density $\varphi_0$ is continuous
    at each $x \in V$.
  
  Denote by $\Pi$ the set of all pairs $(x, k) \in V \times
  \mathbbm{N}$ such that $T^{k + 1} x \in V$ and $T^j x \nin
  V$ for $j = 1, \ldots, k$. It is easy to prove that
  \[ \theta = 1 - \sum_{(x, k) \in \Pi} \frac1{|(T^k)'(x)|}. \]
  It follows that $\theta < 1$ if and only if the set $\Pi$ is nonempty. See also the related examples in \cite{KL2009}.
  \end{remark}

\begin{remark}
  If the holes shrink to a lower dimensional submanifold which is invariant
  under the dynamics, then typically $q_0 > 0$ (and hence $\theta < 1$), see
  {\cite[section 4]{KL2009}} for a particular example.
\end{remark}

\subsection{A sharp error term for exponential hitting time
distributions}\label{subsec:sharp}

Abadi {\cite{abadi2004}} studied the convergence of hitting time distributions
for hitting cylinder sets in a shift space over a finite alphabet $A$ under
$\phi$-mixing stationary laws. He proved: Let $\mu_0$ be a $\phi$-mixing
stationary measure on $A^{\mathbbm{N}}$ with $\sum_n \phi (n) < \infty$. Then
there are constants $C > 0$ and $0 < \Xi_1 < 1 < \Xi_2 < \infty$ such that for
all $L \in \mathbbm{N}$, cylinder sets $A = [a_0 \ldots a_{L - 1}]$ and $t
\geqslant 1$, there exists $\xi_A \in [\Xi_1, \Xi_2]$ with the following
property: for the first hitting time $\tau_A$ of $A$ and all $t > 0$,
\begin{equation}
  \left| \mu_0 \left\{ \tau_A > \frac{t}{\xi_A \mu_0 (A)} \right\} - e^{- t} \right|
  \leqslant C \delta (A) (t \vee 1) e^{- t} \label{eq:abadi1}
\end{equation}
where $\delta (A) = \min_{L \leqslant N \leqslant 1 / \mu_0 (A)} (N \mu_0 (A)
+ \phi (N))$.

An analogous result follows in the dynamical REPFO setting. One should
observe, however, that Abadi assumes only a summable mixing rate whereas the
REPFO setting always implies exponential decay of correlations for ``regular''
observables (i.e. for observables in $V$).

\begin{proposition}
  Suppose that in a REPFO setting $\mu_0 (A_{\varepsilon}) > 0$ for all
  $\varepsilon \in E$ and that the number $\theta$ defined in
  (\ref{eq:KL2009b}) is strictly positive. Then there is a constant $C > 0$
  such that for all $\varepsilon \in E$ sufficiently close to $0$ there exists
  $\xi_{\varepsilon} > 0$ with the following property: for the first hitting
  time $\tau_{\varepsilon}$ of $A_{\varepsilon}$ and all $t > 0$:
  \begin{equation}
    \left| \mu_0 \left\{ \tau_{\varepsilon} \geqslant \frac{t}{\xi_{\varepsilon}
    \mu_0 (A_{\varepsilon})} \right\} - e^{- t} \right| \leqslant C
    \delta_{\varepsilon} (t \vee 1) e^{- t} \label{eq:abadi}
  \end{equation}
  where $\delta_{\varepsilon} = \min_{1 \leqslant N} (N \eta_{\varepsilon} +
  \kappa_N) =\mathcal{O}(\eta_{\varepsilon} \log \eta_{\varepsilon})$. The
  numbers $\xi_{\varepsilon}$ satisfy $\lim_{\varepsilon \rightarrow 0}
  \xi_{\varepsilon} = \theta$.
\end{proposition}

\begin{proof}
  Given $\varepsilon > 0$, fix $N \in \mathbbm{N}$ such that $N
  \eta_{\varepsilon} +\mathcal{O}(\kappa_N)$ becomes minimal, and denote this
  quantity by $\delta_{\varepsilon}$. Here $\mathcal{O}(\kappa_N)=\mathcal{O}((1-\gamma)^N)$ is the term
  from equation (\ref{eq:KL09}). Then $N =\mathcal{O}(\log
  \eta_{\varepsilon})$ and $\delta_{\varepsilon}
  =\mathcal{O}(\eta_{\varepsilon} \log \eta_{\varepsilon})$. Let
  $\xi_{\varepsilon} = \theta_{N, \varepsilon} +\mathcal{O}(\kappa_N)$. Then
  \begin{equation}
    \lambda_{\varepsilon} = 1 - \mu_0 (A_{\varepsilon}) \xi_{\varepsilon} 
    \left( 1 + \mathcal{O} (N \eta_{\varepsilon}) \right) = \exp [- \mu_0
    (A_{\varepsilon}) \xi_{\varepsilon}  \left( 1 + \mathcal{O}
    (\delta_{\varepsilon}) \right)] \label{eq:lambda1}
  \end{equation}
  as $\mu_0 (A_{\varepsilon}) \leqslant \| \varphi_0 \| \cdot
  \eta_{\varepsilon}$.
  
  On the other hand, $\lambda_{\varepsilon} = 1 - \mu_0 (A_{\varepsilon})
  (\theta + o(1))$ when $\varepsilon \rightarrow 0$ by
  (\ref{eq:KL2009b}). Therefore
  \begin{displaymath}
    (\theta - \xi_{\varepsilon}) (1 +\mathcal{O}(\delta_{\varepsilon}))  = 
    \theta\, \mathcal{O}(\delta_{\varepsilon}) + o(1) \rightarrow 0
  \end{displaymath}
  so that $\lim_{\varepsilon \rightarrow 0} \xi_{\varepsilon} = \theta > 0$.
  In particular $\xi_{\varepsilon} > 0$ if $\varepsilon$ is sufficiently close
  to $0$.
  
  Now let $n = \lfloor t / (\xi_{\varepsilon} \mu_0 (A_{\varepsilon}))
  \rfloor$ and assume for the moment that $n > | \frac{\log \eta_{\varepsilon}}{\log(1-\gamma)} |$.
  Then, in view of (\ref{eq:pineps}) and (\ref{eq:kappaN}), there is some $q
  \in [0, 1]$ such that
  \begin{eqnarray}
    \mu_0 \left\{ \tau_{\varepsilon} \geqslant \frac{t}{\xi_{\varepsilon}
    \mu_0 (A_{\varepsilon})} \right\} & = & \lambda_{\varepsilon}^{n + q} 
    \left[ \nu_{\varepsilon} (\varphi_0) +\mathcal{O}(\|Q_{\varepsilon}^n \|)
    \right]\text{ from
  (\ref{eq:KL09}).}  \\
    & = & \exp [- t (1 + \mathcal{O} (\delta_{\varepsilon}))] \cdot (1
    +\mathcal{O}(\eta_{\varepsilon})) . \nonumber
  \end{eqnarray}
  (Here we also used $\nu_{\varepsilon} (\varphi_0) = 1
  +\mathcal{O}(\eta_{\varepsilon})$, see {\cite[Lemma 6.1]{KL2009}}.)
  Therefore,
  \begin{equation}
    \left| \mu_0 \left\{ \tau_{\varepsilon} \geqslant \frac{t}{\xi_{\varepsilon}
    \mu_0 (A_{\varepsilon})} \right\} - e^{- t} \right| \leqslant C (t \vee 1) e^{-
    t} \delta_{\varepsilon}  \label{eq:wie abadi} .
  \end{equation}
  (For $t\,\delta_\varepsilon\leqslant c$, where $c>0$ a suitable small constant, one can  estimate by $C\,e^{-t}\,\delta_\varepsilon$, in the other case by $c^{-1}\,C\,t\,e^{-t}\,\delta_\varepsilon$.)

  For $n \leqslant | \frac{\log \eta_{\varepsilon}}{\log(1-\gamma)} |$ we have the following trivial
  estimates:
  \[ 1 \geqslant \mu_0 \left\{ \tau_{\varepsilon} \geqslant
     \frac{t}{\xi_{\varepsilon} \mu_0 (A_{\varepsilon})} \right\} \geqslant
     \mu_0 \{\tau_{\varepsilon} \geqslant n + 1\} \geqslant 1 - \sum_{i = 0}^n
     \mu_0 (T^{- i} A_{\varepsilon}) = 1 - (n + 1) \mu_0 (A_{\varepsilon})
     \geqslant 1 -\mathcal{O}(\delta_{\varepsilon}) \]
  and
  \[ 1 \geqslant e^{- t} \geqslant 1 - t \geqslant 1 - (n + 1)
     \xi_{\varepsilon} \mu_0 (A_{\varepsilon}) \geqslant 1
     -\mathcal{O}(\delta_{\varepsilon}) . \]
  Hence the l.h.s. in (\ref{eq:wie abadi}) is bounded by
  $\mathcal{O}(\delta_{\varepsilon})$, and as $e^{- t} \geqslant 1
  -\mathcal{O}(\delta_{\varepsilon})$, this proves (\ref{eq:wie abadi}) also
  for $n \leqslant | \frac{\log \eta_{\varepsilon}}{\log(1-\gamma)} |$.
\end{proof}

\begin{remark}
  The problem of convergence of the distributions of rescaled hitting times to
  an exponential without an error term as in (\ref{eq:abadi}) was addressed in
  many papers and under various assumptions, mostly for holes
  $A_{\varepsilon}$ that are not necessarily cylinder sets. Galves and Schmitt
  {\cite{GS1990}} started these investigations under an exponential
  $\psi$-mixing assumption. Hirata obtained a slightly weaker form of such
  results for Gibbs states on subshifts of finite type and for Axiom A
  diffeomorphisms {\cite{hirata1993}} and for system with a finite generating
  $\varphi$- mixing partition {\cite{hirata1994}} as an important step in
  proving Poisson laws for these systems. For $\alpha$-mixing prcesses over a
  finite or countable alphabet, Abadi and Saussol {\cite{AS2010}} proved
  convergence of the distributions of the rescaled hitting times to an
  exponential without any further assumption on the mixing rate. Abadi
  {\cite{abadi2006}} discussed in some detail the nature of the scaling
  factors $\xi_A$.
\end{remark}

\subsection{Extreme value statistics and the extremal
index}\label{subsec:extreme}

In a series of papers, Freitas, Freitas and Todd
{\cite{FFT2009,FFT2010b,FFT2010a}} studied extreme value statistics for
observations generated by chaotic dynamical systems and related them to first
hitting time statistics. The starting point of their approach are mixing
conditions in the tradition of Leadbetter
{\cite{leadbetter1973,leadbetter1983}} that are adapted to the problem of
studying extremal events. They show that these conditions are in particular
satisfied if the underlying system is a (topologically mixing) one-dimensional
map of Rychlik type {\cite{rychlik83}} and $\nu_0$ and $\mu_0$ are the
associated conformal resp. invariant measure. Here we show how some of their
results can be derived (and slightly extended) from our formula
(\ref{eq:KL2009b}).

We begin recalling the approach to extreme value laws taken in
{\cite{FFT2009,FFT2010b,FFT2010a}}. Suppose that $X : M \rightarrow
\mathbbm{R}$ is a continuous function. Let $z_{\max} \assign \sup \{z \in
\mathbbm{R}: \nu_0 (X \geqslant z) > 0\}$ and assume that for some $z_0 <
z_{\max}$ all sets $V_z \assign \{X > z\}$ for $z \in [z_0, z_{\max}]$ satisfy
the assumptions on the sets $A_{\varepsilon}$ from above. Suppose for
simplicity that $\nu_0 (V_{z_{\max}}) = 0$. Fix $t > 0$ and determine real
numbers $z_n \in [z_0, z_{\max}]$ such that $\lim_{n \rightarrow \infty} n
\mu_0 \{X > z_n \}= t$. Consider the process $X_n = X \circ T^n$ on the
probability space $(M, \mathcal{B}, \nu_0)$. Then $\max (X_0 (\omega), \ldots,
X_{n - 1} (\omega)) \leqslant z_n$ if and only if $T^i \omega \in M \setminus
V_{z_n}$ for $i = 0, \ldots, n - 1$. Denote $\varepsilon=\varepsilon_n:=z_{\max}-z_n$. Then equations (\ref{eq:pineps}) and
(\ref{eq:lambdaeps}) imply the following {\tmem{extreme value law}}
\begin{eqnarray}
  \nu_0 \{\max (X_0, \ldots, X_{n - 1}) \leqslant z_n \} & = & \nu_0 \{\omega
  \in M : T^i \omega \nin V_{z_n} \text{ for } i = 0, \ldots, n - 1\}
  \nonumber\\
  & = & \nu_0 \{\tau_{\varepsilon} \geqslant n\}  \label{eq:EVL}\\
  & = & \exp (- t \theta (1 + o (1))) \nonumber
\end{eqnarray}
in the limit $n \rightarrow \infty$, where also the fact that
$\lim_{\varepsilon \rightarrow 0} \nu_{\varepsilon} (1) = \nu_0 (1) = 1$ (see
{\cite{KL1999}}) is used. In this context the constant $\theta$ defined in
(\ref{eq:xi}) is called the {\tmem{extremal index}}. It follows from Remark
\ref{remark:periodic1} that in case $V_{z_{\max}} =\{x\}$ is a singleton,
$\theta < 1$ is possible only if $x$ is a periodic point. In view of the
results collected in section \ref{sec:review}, this answers a question by
Freitas, Freitas and Todd from the introduction of {\cite{FFT2010b}}. 


  
  

\section{A review of open systems to which the eigenvalue perturbation formula
applies\label{sec:review}}

In this section we discuss some classes of open systems of type REPFO. The
results of section \ref{sec:main} apply to these systems.

\subsection{Piecewise monotone interval maps of Rychlik type}\label{subsec:rychlik}

Rychlik {\cite{rychlik83}} provided a very elegant and useful scheme to treat
piecewise monotone one-dimensional maps $T : [0, 1] \rightarrow [0, 1]$. The
ingredients of his approach are a non-atomic probability measure
$\nu_0$ on $[0, 1]$ with full support and a function $g : [0, 1] \rightarrow
[0, + \infty)$ of bounded variation. It is assumed that
\begin{enumeratealpha}
  \item there is an open set $U \subset [0, 1]$ with $\nu_0 (U) = 1$ such that
  for each connected component $J$ of $U$, the restriction $T_{|J}$ is a
  homeomorphism from $J$ onto $T J$ (denote the collection of open components
  of $U$ by $\beta$),
  
  \item $g_{|S} \equiv 0$ where $S \assign [0, 1] \setminus U$,
  
  \item $\vartheta \assign \lim_{n \rightarrow \infty} \|g_n \|_{\infty}^{1 /
  n} < 1$ where $g_n = g \cdot g \circ T \cdot \ldots \cdot g \circ T^{n -
  1}$, and
  
  \item $\nu_0 (P f) = \nu_0 (f)$ for each bounded measurable $f : [0, 1]
  \rightarrow \mathbbm{R}$ where
  \[ P f (x) = \sum_{y \in T^{- 1} x} g (y) f (y) . \]
\end{enumeratealpha}
Denote by $B V$ the space of functions of bounded variation on $[0, 1]$ and
define the seminorm $|f|_w = \int |f| d \nu_0$ and the norm $\|f\|=
\tmop{Var} (f) + |f|_w$ on $B V$. Then assumption d) implies (\ref{eq:B1}) for
$\varepsilon = 0$, and Proposition 1 of {\cite{rychlik83}} shows that
\begin{equation}
  \forall \alpha \in (\vartheta, 1) \exists C > 0 \forall f \in B V \forall n
  \in \mathbbm{N}: \sum_{B \in \beta^n} \tmop{Var} (P^n (f 1_B)) \leqslant C
  \kappa^n \tmop{Var} (f) + C \int |f| d \nu_0 \label{eq:rychlik}
\end{equation}
where $\beta^n = \beta \vee T^{- 1} \beta \vee \ldots \vee T^{- (n - 1)}
\beta$. This is (\ref{eq:B2}) for $\varepsilon = 0$. Finally, (\ref{eq:B3})
for $\varepsilon = 0$ follows from {\cite[Theorem 1]{rychlik83}}.

Let $z \in (0, 1)$ and consider intervals $A_{\varepsilon} = (z - \varepsilon,
z + \varepsilon) \subset (0, 1)$. If we replace $S$ by $S \cup \{z\}$, all of
the above assumptions a) -- d) remain valid. So w.l.o.g. $z \in S$. The rare
event PF operators $P_{\varepsilon}$ satisfy
\begin{equation}
  |P_{\varepsilon}^n f|_w = \int_{\bigcap_{k = 0}^{n - 1} T^{- k} ([0, 1\}
  \setminus A_{\varepsilon})} |f| d \nu_0 \leqslant \int |f| d \nu_0 = |f|_w\,,
  \label{eq:rychlik2}
\end{equation}
and, observing that $B \cap \bigcap_{k = 0}^{n - 1} T^{- k} ([0, 1\} \setminus
A_{\varepsilon})$ is a subinterval (possibly empty) of $B$ for each $B \in
\beta^n$, also
\begin{equation}
  \tmop{Var} (P_{\varepsilon}^n f) \leqslant \sum_{B \in \beta^n} \tmop{Var}
  (P_{\varepsilon}^n (f 1_B)) \leqslant \sum_{B \in \beta^n} \tmop{Var} (P^n
  (f 1_B)) \leqslant C \kappa^n \tmop{Var} (f) + C|f|_w \label{eq:rychlik3} .
\end{equation}
The uniform estimates (\ref{eq:rychlik2}) and (\ref{eq:rychlik3}) show that
(\ref{eq:B3}) -- (\ref{eq:B2}) are satisfied for all $\varepsilon > 0$. We
turn to (\ref{eq:B4}):
\[ |P_{\varepsilon} f - P f|_w \leqslant \int_{A_{\varepsilon}} |f| d \nu_0
   \leqslant \nu_0 (A_{\varepsilon})\|f\|_{\infty} \leqslant \nu_0
   (A_{\varepsilon})\|f\|. \]
So $\tau_{\varepsilon} = \nu_0 (A_{\varepsilon}) \searrow 0$ as $\varepsilon
\rightarrow 0$, and this is (\ref{eq:B4}). Thus assumption REPFO.i is
verified.

Assume now that $T$ is mixing in the sense of (\ref{eq:B5}) with unique
invariant density $\varphi_0$. As $\varphi_0$ is of bounded variation, the
following extra assumption is not very restrictive:
\\[2mm]
{\tmstrong{Positivity assumption:}} \ $\delta \assign \lim_{\varepsilon
\rightarrow 0} \max \{\inf_{(z - \varepsilon, z)} \varphi_0, \inf_{(z, z +
\varepsilon)} \varphi_0 \}> 0$.
\\[2mm]
Indeed, if $T$ has only finitely many monotone branches, Theorem 8.2.3 in
{\cite{BG-book}} shows that this assumption holds if $z$ belongs to the
support of the invariant measure $f_0 \nu_0$.

Under the positivity assumption,
\[ \|1_{A_{\varepsilon}} f\|= \int_{A_{\varepsilon}} |f| d \nu_0 + \tmop{Var}
   (1_{A_{\varepsilon}} f) \leqslant |f|_w + 2\|f\|_{\infty} + \tmop{Var} (f)
   \leqslant 3\|f\| \]
and
\[ | \nu_0 (1_{A_{\varepsilon}} f) | \cdot \|1_{A_{\varepsilon}} \varphi_0 \|
   \leqslant \nu_0 (A_{\varepsilon})\|f\|_{\infty} \cdot 3\,\| \varphi_0 \| \leqslant 4 \delta^{- 1} \|
   \varphi_0 \| \cdot \nu_0 (1_{A_{\varepsilon}} \varphi_0) \cdot \|f\| \]
for all suficiently small $\varepsilon > 0$. These two estimates yield REPFO.ii and iii.

\begin{remark}
  The class of Rychlik maps contains the classes of maps considered in
  {\cite{BB2010}} and {\cite{LM2003}}. Freitas et al. {\cite{FFT2010b}} proved
  the validity of an extreme value law (\ref{eq:EVL}) for Rychlik maps along
  the classical probabilistic lines initiated by Leadbetter
  {\cite{leadbetter1973,leadbetter1983}}. Exponential statistics for Rychlik
  maps around $\mu_0$-a.e. point $z$ were already derived in
  {\cite{BSTV2003}}.
\end{remark}

See also Remark \ref{remark:rychlik1} for an explicit formula for $\theta$.

\subsection{Piecewise expanding maps in higher dimensions}

For piecewise expanding maps in higher dimensions, the classical BV-space is
the first candidate for a suitable function space $V$
{\cite{GB1989,keller79,thomine2011}}. For the present purpose, however, the
quasi H\"older spaces of Blank {\cite{blank87}} and Saussol {\cite{saussol00}}
are more suited, because, in contrast to classical BV-spaces, these spaces are
algebras and their norms control the $L^{\infty}$-norm {\cite[Proposition
3.4]{saussol00}}.

Without going into details, we consider a piecewise $C^{1 + \alpha}$ map on a
compact set $\Omega \subset \mathbbm{R}^N$ for which expansion wins over the
effects of discontinuities in the sense of {\cite[condition
(PE5)]{saussol00}}. The corresponding PF operator $P$ satisfies (\ref{eq:B3})
-- (\ref{eq:B2}) for $\varepsilon = 0$ in view of {\cite[Lemma 4.1 and Theorem
5.1]{saussol00}}. Fix a point $z$ in the interior of the domain $U_i$ of one
of the branches of $T$, and denote by $A_{\varepsilon}$ the
$\varepsilon$-neighborhood of $z$ (w.r.t. some $\ell_p$-norm in
$\mathbbm{R}^N$). Then the corresponding rare event PF operators
$P_{\varepsilon}$ satisfy the same estimates that $P$ satisfies if the domain
$U_i$ is replaced by $U_i \setminus A_{\varepsilon}$ (for sufficiently small
$\varepsilon$). This modification of $U_i$ increases the constant $G$ in (PE5)
of {\cite{saussol00}}{\footnote{The constant there is denoted $G (\varepsilon,
\varepsilon_0)$; note that our $\varepsilon$ (size of the hole) and the
$\varepsilon$ in $G (\varepsilon, \varepsilon_0)$ are unrelated.}} by at most
$\tmop{const} \varepsilon^{N - 1}$, and it follows from the explicit
expressions for the constants in (\ref{eq:B2}) that are given in the proof of
Lemma 4.1 in {\cite{saussol00}}, that (\ref{eq:B3}) -- (\ref{eq:B2}) are
indeed valid for all sufficiently small $\varepsilon > 0$. Assumption
(\ref{eq:B4}) is satisfied because $\|f\|_{\infty} \leqslant \tmop{const}
\|f\|$ where $\|.\|$ is Saussol's quasi H\"older norm. This is REPFO.i. As
$\|1_{A_{\varepsilon}} \|$ is of the order of $\varepsilon^{N - 1}$, REPFO.ii
follows from $\|1_{A_{\varepsilon}} f\| \leqslant \tmop{const}
\|1_{A_{\varepsilon}} \| \cdot \|f\|$, see {\cite[Proposition
3.4]{saussol00}}. Under the additional {\tmstrong{positivity assumption}}
$\delta \assign \lim_{\varepsilon \rightarrow 0} \inf_{A_{\varepsilon}}
\varphi_0 > 0$, condition REPFO.iii follows immediately from $\|f\|_{\infty}
\leqslant \tmop{const} \|f\|$.

In higher dimensional systems, holes that do not shrink to a point may occur
quite naturally, for example when $A_{\varepsilon}$ is defined as the event
that a state vector has two nearly (i.e.up to $\varepsilon$) identical
coordinates. A simple example of this type, namely a direct product of a
piecewise expanding interval map $T$ with itself, is discussed in
{\cite{CC1994}}, where exponential hitting time statistics are proved (as a
corollary to a Poissonian limit law) and where a (complicated) formula for the
exponent of the limit distribution is derived also in the case where only some
power $T^m$ of the map is uniformly expanding. In {\cite[section 4]{KL2009}}
this formula is rederived from spectral perturbation theory when $T$ itself is
expanding. Indeed, for $m > 1$, too, this formula can be derived from
(\ref{eq:KL2009b}) by a straightforward, though lengthy calculation. In the
same reference also a modification of the product system by some weak coupling
is discussed.

\subsection{Gibbs measures on subshifts of finite type}

Ferguson and Pollicott {\cite{FP2010}} study escape rates for Gibbs measures
with H\"older continuous potentials on subshifts of finite type. They obtain
formula (\ref{eq:KL2009b}) for suitable holes (not just cylinder sets) in this
context and use it to derive from that the corresponding formula for the
escape from conformal repellors relative to equilibrium states, and also a
related formula for the Hausdorff dimension of the set of points on the
repellor that never enter the hole. For deriving formula (\ref{eq:KL2009b})
they produce a function space and two norms fitting the REPFO setting.

\subsection{Topological entropy of perturbations of shifts of finite type}

If a subshift of finite type over a finite alphabet is perturbed by deleting
one or several blocks and if these blocks have small mass under the measure
$\mu_0$ of maximal entropy of the unperturbed subshift, the topological
entropy of the systems drops by a small amount. As the topological entropy is
the logarithm of the leading eigenvalue of the corresponding transfer operator
with constant weight $1$, the amount by which the entropy drops down is
controlled by formula (\ref{eq:KL2009b}) when it is applied to the operators
normalized such that the unperturbed operator has leading eigenvalue $1$.

Lind {\cite[Theorem 3]{lind1989}} provides upper and lower bounds on the
entropy difference when a single (long) block is discarded. These estimates
can be made more precise using formula (\ref{eq:KL2009b}). Details of this in
two slightly different settings can be found in {\cite[section 5.2]{KL2009}}
and in {\cite[Corollary 5.4]{FP2010}}.

\subsection{Further candidates}

The investigation of statistical properties of dynamical systems via the
spectral properties of their associated Perron Frobenius operators is
currently a very active area of research. Some of the available results apply
to piecewise smooth systems and there is good hope that these systems are of
REPFO type. Here we discuss briefly the situation for piecewise hyperbolic
maps and for coupled map lattices of piecewise expanding 1D maps.

\paragraph{Piecewise hyperbolic maps: }Demers and Liverani {\cite{DL2008}}
studied Perron-Frobenius operators for a broad class of piecewise hyperbolic
maps on two-dimensional compact Riemannian manifolds. They provide a normed
space $(\mathcal{B}, \|.\|)$ of distributions and prove in {\cite[Lemmas 3.5
and 6.5, Proposition 6.6]{DL2008}} that assumptions (\ref{eq:B3}) --
(\ref{eq:B4}) hold for many reasonably regular holes $A_{\varepsilon}$. In
particular this is true for disks of radius $\varepsilon$ when $\varepsilon$
is sufficiently small (so that the curvature of $\partial A_{\varepsilon}$ is
large, see {\cite[Remark 2.16]{DL2008}}). Hence REPFO.i is satisfied. A first
step towards a proof of REPFO.iii is the observation that $| \nu_0
(1_{A_{\varepsilon}} f) | \leqslant \varepsilon^{1 + \alpha} \|f\|_s \leqslant
\|f\|$ for all $f \in V$ according to the definition of the stable norm
$\|.\|_s$ and the strong norm $\|.\|$ in {\cite[eq. (2.3)and (2.5)]{DL2008}}.
Therefore, it mainly remains to show that $\varepsilon^{1 + \alpha}
\|1_{A_{\varepsilon}} \varphi_0 \| \leqslant \nu_0 (1_{A_{\varepsilon}}
\varphi_0)$.

Baladi and Gou\"ezel {\cite{BG2009,BG2010}} studied the Perron
Frobenius operator of piecewise hyperbolic systems in any dimension on
certain classes of Triebel spaces which they chose such that
multiplication by a regular set, e.g. a convex set, acts as a bounded
linear operator. This indicates that also their systems might fit into
the REPFO setting.

\paragraph{Coupled map lattice of piecewise expanding interval maps: }Coupled
map lattices (over $\mathbbm{Z}^d$) of piecewise expanding interval maps were
studied in {\cite{KL2006,KL2009a}} using a spectral theoretic approach. Here
$V$ is a space of distributions on $[0, 1]^{\mathbbm{Z}^d}$ that contains in
particular all probability measures whose finite dimensional marginals have
densities of bounded variation with uniform bounds on the variation of all
these marginal densities in all coordinate directions. While {\cite{KL2006}}
treats only coupling mechanisms which are uniformly weak on the whole phase
space, the short note {\cite{KL2009a}} deals with cases where strong (and
discontinuous) coupling effects are allowed on small parts of the phase space.
The estimates used in this context should be helpful also in attempting to
prove REPFO.i-iii.

\end{document}